\documentclass[11pt,a4paper]{amsart}
\usepackage[foot]{amsaddr}
\usepackage{texmac,fancyhdr}
\usepackage{enumerate}
\operators{Stab,Prim,Inn,Out}
\let\prim\Prim
\bbsymbols{b}{C,F,Z,Q}
\calsymbols{c}{H,U,X,Y,P,M}
\begin{document}

\title{Brauer relations in finite groups II -- quasi-elementary groups of order $p^aq$}
\author{Alex Bartel$^1$}
\address{$^1$Mathematics Institute, Zeeman Building, University of Warwick,
Coventry CV4 7AL, UK}
\email{a.bartel@warwick.ac.uk}
\author{Tim Dokchitser$^2$}
\address{$^2$Department of Mathematics, University Walk, Bristol BS8 1TW, United Kingdom}
\email{tim.dokchitser@bristol.ac.uk}
\llap{.\hskip 10cm} \vskip -0.8cm

\begin{abstract}
This is the second in a series of papers investigating the space of Brauer
relations of a finite group, the kernel of the natural map from its 
Burnside ring to the rational representation ring.
The first paper classified all primitive Brauer relations, that
is those that do not come from proper subquotients.
In the case of quasi-elementary groups the description is intricate,
and it does not specify groups that have primitive relations in terms of
generators and relations. In this paper we provide such a classification in
terms of generators and relations for quasi-elementary groups of order $p^aq$.
\end{abstract}
\maketitle
\tableofcontents

\pagestyle{fancy}
\fancyhf{} 
\fancyhead[C]{\textsc{Brauer relations in quasi--elementary groups of order $p^aq$}}
\fancyfoot[C]{\thepage}

\section{Introduction}
Let $G$ be a finite group.
Recall that there is a natural map from the Burnside ring $B(G)$ of $G$
to the rational representation ring $R_\Q(G)$, which sends a $G$-set to
the corresponding permutation representation. In \cite{bra1}, we have
completely  described the kernel $K(G)$ of this map by giving a canonical
set of generators, following the work of Bouc \cite{Bouc} who had described
this kernel for $p$-groups.
The purpose of this paper is to make our description more explicit in the
most difficult case, that of quasi-elementary groups.

We will write elements of $B(G)$ as $\Z$-linear combinations
$\sum_{H\leq G} a_H H$, using the identification between transitive $G$-sets
and conjugacy classes of subgroups of $G$ that maps a $G$-set to a point
stabiliser. We call an element of $K(G)$ a \emph{Brauer relation}.

If $H\leq G$, then there is an induction map $B(H)\longrightarrow B(G)$,
which yields an induction map $K(H)\longrightarrow K(G)$. We say that a
Brauer relation of $G$ is \emph{induced from $H$} if it is in the image of
this map. Also, if $N\normal G$, then there is an inflation map
$B(G/N)\longrightarrow B(G)$, which induces a map $K(G/N)\longrightarrow K(G)$.
We say that a Brauer relation of $G$ is \emph{lifted from $G/N$} if it is
in the image of this inflation map. We call a relation \emph{imprimitive}
if it is a linear combination of relations that are induced and/or lifted
from proper subquotients of $G$ and \emph{primitive} otherwise. The
quotient of $K(G)$ by the subgroup consisting of imprimitive relations will
be denoted by $\Prim(G)$.

The main result of \cite{bra1} is a group theoretic criterion for $\Prim(G)$
to be non-trivial and a determination of the group structure of $\Prim(G)$
and of a set of generators. This criterion is most complicated when $G$ is
quasi-elementary, i.e. when it is of the form $G= C\rtimes P$, where
$C$ is cyclic and $P$ is a $p$-group. Write $K$ for the kernel of
the conjugation action of $P$ on $C$. It is shown in \cite{bra1} that for
such a group to have primitive relations, $C$ must be of square-free order,
and $K$ must be either trivial, or isomorphic to $D_8$,
or have normal $p$-rank one. When $K$ is trivial, a complete description of
$\Prim(G)$ is provided by \cite[Proposition 6.5]{bra1}.
Here, we explicitly describe the structure of those quasi-elementary groups
for which $\Prim(G)$ is non-trivial in the case that the order of $C$ is
prime and $K$ is non-trivial.
We determine the general shape of a presentation for such groups
in terms of generators and relations, thereby explicating the criterion of
\cite[Theorem 7.30]{bra1} in this special case.

\begin{theorem}\label{thm:main}
Let $G= C\rtimes P$ be quasi-elementary, where $C$ is cyclic of prime
order $q$ and $P$ is a $p$-group with $p\neq q$.  Suppose that
$K=\ker(P\rightarrow \Aut C)$ is either isomorphic to $D_8$ or has normal
$p$-rank one. Let $K$ be generated by $c$ of
order $p^n$ and possibly $x$ of order 2 or 4 as in Proposition \ref{prop:prankone},
and define $j,k$ by $hch^{-1} = c^j$, $hxh^{-1} = c^kx$.
 A necessary condition for $G$ to have
primitive relations is that $P$ must be of the form $P=K\rtimes A$,
where $A$ acts faithfully on $C$, and in particular is cyclic of order $p^m$,
say, $A=\langle h\rangle$. Assume that these conditions are satisfied.
The following are all the cases in which $\Prim(G)$ is
non-trivial, together with the group structure of $\Prim(G)$:
\begin{enumerate}
\item\label{item:cyc1} $K$ is cyclic, $p\ne 2$, $n \le m$; $\Prim(G)\cong C_p^{u}$,
where $u=p-2$ if $n=1$ and $u=p-1$ otherwise;~or
\item\label{item:cyc2} $K$ is cyclic, $p=2$, $j \ne -1$; moreover, either
$1 < n \le m$, or $j\equiv 3\pmod 4$ and the order of $j$ in $(\Z/2^n\Z)^\times$
divides $2^{m-1}$; in either case, $\Prim(G)\cong C_2$;~or
\item\label{item:quat} $K$ is generalised quaternion, $k$ is odd, and $n<m$;
$\Prim(G)\cong C_2$;~or
\item\label{item:diheven} $K$ is dihedral, $k$ is even, $j\neq \pm1$, and
\[
2^n|j^{2^{m-1}}-1,\;\; 2^n|k(j^{2^{m-1}}-1)/(j-1);
\]
$\Prim(G)\cong C_2$;~or
\item\label{item:dihodd} $K$ is dihedral, $k$ is odd, and either $m>n$ or
\[
j\neq \pm 1,\;\; 2^n|(j^{2^{m-1}}-1)/(j-1);
\]
$\Prim(G)\cong C_2\times C_2$
if $m>n$ and $j\neq \pm 1$, and $\Prim(G)\cong C_2$ otherwise.
\end{enumerate}
\end{theorem}

Our strategy is equally suited to obtaining a concrete result of the same
nature when $|C|$ has any other fixed number of prime divisors, although in
the case under consideration some substantial simplifications occur, most
notably Proposition \ref{prop:semidir}.

\section{Some general results on the structure of $P$}

We recall the relevant results from \cite{bra1}. Throughout, we assume that
$G= C\rtimes P$, where $C$ is cyclic of prime order $q$ and $P$ is a
$p$-group with $p\neq q$. The kernel of the conjugation action of $P$ on $C$
is denoted by $K$.

\begin{proposition}[\cite{bra1}, Proposition 7.6]
If $G$ has a primitive relation, then $K$ is either trivial or isomorphic to
$D_8$ or has normal $p$-rank one.
\end{proposition}
Recall that the normal $p$-rank of a group is defined as the maximal rank of
a normal elementary abelian $p$-subgroup.

\begin{prop}[\cite{Gor-68}, Theorem 5.4.10]\label{prop:prankone}
Let $X$ be a $p$-group with normal $p$-rank one. Then $X$ is one of the following:
\begin{itemize}
\item the cyclic group $C_{p^n}\!=\!\langle c|c^{p^n}\!=\!1\rangle$;
\item the dihedral group $D_{2^{n+1}}\!=\!\langle c,x|c^{2^n}\!=\!x^2\!=\!1, xcx\!=\!c^{-1}\rangle$ with $n\geq 3$;
\item the generalised quaternion group, $Q_{2^{n+1}} \!=\! \langle c,x|c^{2^{n-1}}\!=\!x^2,x^{-1}cx \!=\! c^{-1}\rangle$ with $n\geq 2$;
\item the semi-dihedral group $SD_{2^{n+1}}\!=\!\langle c,x|c^{2^n} \!=\! x^2 \!=\! 1, xcx\!=\!c^{2^{n-1}-1}\rangle$ with $n\geq 3$.
\end{itemize}
\end{prop}

\begin{notation}
If $K$ is trivial, then $\Prim(G)$ is described explicitly by
\cite[Proposition 6.5]{bra1}. From now on, assume that $K$ is either
isomorphic to $D_8$ or has normal $p$-rank one. In particular, $K$ contains
a unique subgroup of order $p$ that is normal in $G$, which we will denote by
$\centralCp$. Let $\cH$ denote the set of those subgroups\footnote{This
definition slightly differs from the one
in \cite{bra1}, where $\cH$ was defined as a set of representatives of
conjugacy classes of subgroups.}
of $P$ that do not
contain $\centralCp$, and let $\cH_m$ be the set of elements of $\cH$ of
maximal size. Let $C_K$ be either $K$ if $K$ is cyclic, or a cyclic index
two subgroup of $K$ that is normal in $G$ otherwise (such an index two 
subgroup is unique unless $K=Q_8$, in which case there may be three such
groups, see \cite[Notation 7.9]{bra1} for details). Set $\bigC=CC_K$.
\end{notation}

\begin{lemma}[\cite{bra1}, Lemma 7.27]\label{lem:genKnontriv}
The group $\Prim(G)$ is generated by relations of the form
\[
\Theta=\sum_{\tilde{C}\leq \bigC} \mu(|\tilde{C}|)(\tilde{C}H-\tilde{C}H'),
\]
for $H,H'\in \cH_m$, where $\mu$ denotes the Moebius function.

\end{lemma}

\begin{proposition}\label{prop:semidir}
If $G$ has a primitive relation, then $P$ is a semi-direct product by $K$,
$P=K\rtimes A$, where $A$ acts faithfully on $C$.
\end{proposition}
\begin{proof}
Since $|C|$ is prime, $\Aut C$ is cyclic. So for a subgroup $H$ of $P$,
either the image of $H$ under $P\rightarrow \Aut C$ is equal to that of
$P$, or $H$ is contained in the pre-image under $P\rightarrow \Aut C$
of the unique index $p$ subgroup of $\Aut C$. Thus, either $\cH_m$
contains a subgroup whose image under $P\rightarrow\Aut C$ is equal to
that of $P$, or the relations of Lemma \ref{lem:genKnontriv} are all
imprimitive by \cite[Proposition 3.7]{bra1}.

Suppose for the rest of the proof that $\cH$ contains a subgroup $A$ of $P$
whose image in $\Aut C$ is equal
to  that of $P$. If $A$ intersects $K$ trivially, then $P$ must be a
semi-direct product $K\rtimes A$. In particular, this proves the claim 
in the cases that $K$ is cyclic or generalised quaternion, since in those
cases intersecting $K$ trivially is equivalent to not containing $\centralCp$.

Suppose that
$K$ is dihedral or semi-dihedral (in particular, $P$ is a 2-group).
Then $P$ is a semi-direct product by $C_K$. But since the automorphism of
$C_K$ given by conjugation by a non-central involution of $K$ is not
divisible by 2 in $\Aut C_K$, this implies that $P$ is a semi-direct product
by $K$ and the proof is complete.
\end{proof}

\begin{lemma}[\cite{bra1}, Corollary 7.4]\label{lem:ZP}
If $K$ is non-trivial and $P$ has cyclic centre, then $G$ has no primitive
relations.
\end{lemma}

\begin{hypothesis}\label{hyp:nontrivker}
From now on, assume that $P$ is a semi-direct product by $K$ and that the
conjugation action of the quotient on $K$ has a non-trivial kernel.
\end{hypothesis}
Note that if $|K|=p$, and in particular $P$ is a direct product, then
\cite[Proposition 7.29]{bra1} implies that $\Prim(G)\cong\F_p^{p-2}$.
In the case that $|K|\neq p$, we recall the description of $\Prim(G)$ in the
special case under consideration that we
want to make explicit.

\begin{thm}[\cite{bra1}, Theorem 7.30]\label{thm:Knontriv}
Assume that $|K| > p$. Define
a graph $\Gamma$ whose vertices are the elements of $\cH_m$ and with an edge
between $H,H'\in \cH_m$ if one of the following applies:
\begin{enumerate}
\item the subgroup generated by $H$ and $H'$ is a proper subgroup of $P$;
\item\label{edge:quo} the intersection $H\cap H'$ is of index $p$
in $H$ and in $H'$, and $HH'/H\cap H'$ is either dihedral, or the Heisenberg 
group of order $p^3$.
\end{enumerate}
Let $d$ be the number of connected components of $\Gamma$.
Then $\Prim(G)\!\iso\!(C_p)^{d-1}$, generated by relations
$\Theta=\sum_{\tilde{C}\leq \bigC} \mu(|\tilde{C}|)(\tilde{C}H-\tilde{C}H')$
for $H,H'\in\cH_m$ corresponding to
distinct connected components of the graph.
\end{thm}

\begin{remark}\label{rmrk:conjugateGroups}
If $H$ is a proper subgroup of $P$, then its image in the Frattini quotient
of $P$ is also a proper subgroup. Since the Frattini quotient is abelian,
$P$ acts trivially on it by conjugation, so if $H'$ is any conjugate of
$H$, $H$ and $H'$ never generate the whole of $P$. Thus, conjugate subgroups
always lie in the same connected component of $\Gamma$.
\end{remark}
\begin{notation}
We will retain the notation $\Gamma$ for the graph described in Theorem
\ref{thm:Knontriv} for the rest of the paper.
\end{notation}
\section{$K$ is cyclic}

Let 
\[
  K=\langle c|c^{p^n}=1\rangle,\;n\geq 2.
\]
The only subgroup of $K$ that does not contain $\centralCp$ is the trivial
group, so $\cH$ consists of subgroups of $P$ that intersect $K$ trivially.
We have $P=K\rtimes A$ where
$A=\langle h\rangle$ is cyclic of order $p^m$, acts faithfully on $C$ and
acts as $hch^{-1} = c^j$ on $K$ with $j$ of order dividing $p^{m-1}$ in
$(\bZ/p^n\bZ)^\times$ and in particular $j\equiv 1 \pmod p$.

All elements of $\cH_m$ are of the form $H_\alpha=\langle c^\alpha h\rangle$
for $0\leq \alpha \leq p^n-1$. Conversely, such an $H_\alpha$ is a
complement of $C\times K$ in $G$ if and only if $H_\alpha\cap K=\{1\}$ if
and only if $c^\alpha h$ has order $p^m$. 

\begin{lem}\label{lem:ladicVal}
Let $l$ be a prime and let $n\equiv 1 \pmod l$ if $l$ is odd and
$n\equiv 1 \pmod 4$ if $l=2$. Then
\[v_l(n^s - 1) = v_l(n - 1) + v_l(s)\]
for any positive integer $s$, where $v_l$ is the normalised $l$-adic
valuation.
\end{lem}

\begin{proof}
Set $r=l$ if $l$ is odd and $r=4$ if $l=2$. The isomorphism
$$
\exp: (r\bZ_l,+)\rightarrow (1 + r\bZ_l,\times)$$
preserves the standard filtrations on
both hand sides, i.e. it satisfies $v_l(t) = v_l(\exp(t)-1)$.
Raising to the $s$-th power on the right hand side corresponds to 
multiplying by $s$ on the left hand side and the result follows.
\end{proof}

Now,
\[
(c^\alpha h)^{p^m} = c^{\alpha(j^{p^m-1}+j^{p^m-2}+\ldots+j+1)}\, h^{p^m} =
\bigleftchoice{c^{\alpha\frac{j^{p^m}-1}{j-1}}}{j\neq 1}{c^{\alpha p^m}}{j=1}.
\]
It follows that if $\alpha\equiv \beta\pmod p$, then
$H_\alpha\in \cH_m$ if and only if $H_\beta\in\cH_m$. Moreover,
two such groups generate a proper subgroup of $P$, so the
corresponding vertices of $\Gamma$ lie in one connected component of the
graph. Thus, each
connected component is represented by some $H_\alpha$ with $0\leq \alpha\leq p-1$.
If $p$ is odd, Lemma \ref{lem:ladicVal} implies that for any
$\alpha\in\{1,\ldots p-1\}$, $H_\alpha$ is a complement
if and only if $m\geq n$. If on the other
hand $p=2$, then $H_1$ is a complement if and only if either 
$j\equiv 3 \pmod 4$ (since the order of $j\in(\bZ/2^n\bZ)^\times$ divides
$2^{m-1}$) or $j\equiv 1 \pmod 4$ and $m\geq n$ (by Lemma \ref{lem:ladicVal}).
In particular, if $p$ is odd and $n>m$ or if 
$p=2$, $j\equiv 1\pmod 4$ and $n>m$, then the graph $\Gamma$ has only one
connected component, so $G$ has no primitive relations.

From now on assume that either $n\le m$ or $p=2$, $j\equiv 3\pmod 4$, and 
the order of $j$ in $(\Z/2^n\Z)^\times$ divides $2^{m-1}$.

If $\alpha \not\equiv \beta\pmod{p}$, then $H_\alpha$ and $H_\beta$
together generate $P$. Since $n>1$, $P/H_\alpha\cap H_\beta$ cannot be
isomorphic to the Heisenberg group of order $p^3$, having exponent strictly
bigger than $p$. Finally, $P/H_\alpha\cap H_\beta$ can only be isomorphic to
a dihedral group if $p=2$ and $j=-1$, in which case any such quotient is
dihedral. This completes the proof of cases (\ref{item:cyc1}) and (\ref
{item:cyc2}) of Theorem \ref{thm:main}.

In all the remaining cases, we have $p=2$.

\section{$K$ is semi-dihedral}

A presentation of $P$ by generators and relations is given by
\begin{eqnarray*}
  \lefteqn{K\rtimes A  = \langle x,c,h|}\\
  & & x^2=c^{2^n}=h^{2^m}=1, xcx=c^{2^{n-1}-1},
  hxh^{-1} = c^kx, hch^{-1} = c^j\rangle,\\
  & & n\geq 3,\;j\in (\Z/2^n\Z)^\times,\;k\in\Z/2^n\Z.
\end{eqnarray*}
Here, $k$ must be even, since if $z$ is the central involution of $K$, then
\[
  h = x(xhx)x =xz^kc^{-k}hx = z^{k}xc^{-k}c^{k}xh = z^{k}h.
\]
The only conjugacy class of non-trivial subgroups of $K$ that do not contain
$\centralCp = \langle z\rangle$ is that of non-central involutions,
represented by $\langle x\rangle$, so the elements of $\cH$ are subgroups of
$P$ that either intersect $K$ trivially or those that intersect it in
$\langle c^rx\rangle$ for some $r\in \Z/2^n\Z$.

Since $k$ is even, $h$ fixes a non-central involution:
$hc^ixh^{-1} = c^{ki+k}x$, and $i$ can be chosen such that $c^{ki+k}=1$. So
without loss of generality, replacing $x$ by $c^ix$ for such an $i$,
assume that $k=0$. In particular, elements of
$\cH_m$ have size $2^{m+1}$ and are isomorphic to $C_2\times C_{2^m}$.

We claim that all elements of $\cH_m$ are contained in
$\langle c^2,h,x\rangle$, so that any two of them generate a proper subgroup
of $P$, which will show that $G$ has no primitive relations when $K$ is
semi-dihedral. Indeed, any element of $\cH_m$ is of the form
$\langle c^rh, c^sx\rangle$, with the two generators commuting, and it
suffices to show that both $r$ and $s$ must be even.
Since $1=(c^sx)^2=z^s$, $s$ must be even. Moreover, the condition that the
two generators commute implies that
\begin{eqnarray*}
\lefteqn{c^rhc^sx(c^rh)^{-1}(c^sx)^{-1} =
z^{-r}c^{2r+(j-1)s}=1}\\
& & \Rightarrow 2r+(j-1)s\equiv 0\pmod 4\\
& & \Rightarrow r\text{ is even},
\end{eqnarray*}
as required.

\section{$K$ is generalised quaternion}

A presentation for $P$ by generators and relations is
\begin{eqnarray*}
  \lefteqn{P = \langle x,c,h\;|\;c^{2^n}=h^{2^m}=1,}\\
  & & x^2=c^{2^{n-1}}, xcx^{-1}=c^{-1},
  hxh^{-1} = c^kx, hch^{-1} = c^j \rangle,\\
  & & n\geq 2,\;j\in (\Z/2^n\Z)^\times,\;k\in\Z/2^n\Z.
\end{eqnarray*}

Since every non-trivial subgroup of $K$ contains $\centralC2$, $\cH$ consists
of subgroups of $P$ that intersect $K$ trivially.
Note, that it follows from the fact that $j$ is odd
that the parity of $k$ is independent of the choice of $x$.
\newline
\textbf{Case 1: $k$ even}
\newline
Then, $h$, $x$, and $c$ are independent in the Frattini quotient of $P$,
which is therefore three-dimensional. Since elements of $\cH$ are cyclic, no
two of them can generate $P$, so the graph $\Gamma$ has only one connected 
component and $G$ has no primitive relations.
\newline
\textbf{Case 2: $k$ odd}
\newline
The Frattini quotient of $P$ is then
two-dimensional and has exactly three lines, generated by the images of $x$,
$h$, and $xh$, respectively. Only two of these
can be images of elements of $\cH_m$, the lines generated by the
images of $h$ and $xh$. The structure of $K$ implies that
two elements of $\cH_m$ never satisfy condition (\ref{edge:quo}) of
Theorem \ref{thm:Knontriv}, so we deduce that there is no edge between the
vertices of $\Gamma$ corresponding to $H,H'\in\cH_m$ if and only if their 
images in the Frattini quotient of $P$ are the lines generated by $h$ and
$xh$, respectively, and that $\Prim(G)\cong C_2$ if and only if there exists
$H\in \cH_m$ with image $\langle xh\rangle$ in the Frattini quotient, and is
trivial otherwise.

An arbitrary element of $\cH_m$ is of the form $\langle hx^\delta c^r\rangle$
for $\delta\in\{0,1\}$ and $0\leq r\leq 2^m-1$, and a group of this
form is in $\cH$ if and only if $hx^\delta c^r$ has order $2^m$
(in general, the order is greater than or equal to $2^m$).
Such a group has image equal to that of $\langle xh\rangle$ in the
Frattini quotient if and only if $\delta=1$. We have
\[
      (hxc^r)^{2^m} = c^{(j^{2^m-1}+j^{2^m-3}+\ldots+j)(i-r+jr)}
      = \bigleftchoice{c^{\text{odd}\cdot 2^{m-1}}}{j^2=1}{c^{\text{odd}\cdot
      (j^{2^m}-1)/(j^2-1)}}{j^2\ne 1}.
\]
By Lemma \ref{lem:ladicVal}, the 2-adic valuation of the exponent is $m-1$ in
both cases, so the right hand side is 1 if and only if $m-1\geq n$. That
condition also insures that $A$ acts non-faithfully on $K$. This
completes the proof of case (\ref{item:quat}) of Theorem \ref{thm:main}.

\section{$K$ is dihedral}

In this case, $P$ is given by generators and relations as
\begin{eqnarray*}
  \lefteqn{P=\langle x,c,h\;|}\\
  & & c^{2^n}=h^{2^m}=x^2=1, xcx=c^{-1},hxh^{-1}=c^kx,hch^{-1}
  =c^j\rangle,\\
  & & n\geq 2,\;j\in (\Z/2^n\Z)^\times,\;k\in\Z/2^n\Z.
\end{eqnarray*}
We have the following identities, which we will use repeatedly:
\begin{eqnarray}
h^i x h^{-i} &=& c^{k(j^{i-1}+j^{i-2}+\ldots+j+1)}\> x =
    \bigleftchoice{c^{k\frac{j^i-1}{j-1}}x}{j\ne 1}{c^{ki}x}{j=1},
    \label{eq:actonx} \\
  h^i c h^{-i} &=& c^{j^i}.\label{eq:actonc}
\end{eqnarray}
Recall that by Hypothesis \ref{hyp:nontrivker}, $h^{2^{m-1}}$ acts trivially on
$K$. Equations (\ref{eq:actonx}) and (\ref{eq:actonc}) imply that this
condition is equivalent to
\[
  2^n|j^{2^{m-1}}-1\text{ i.e. }2^n|(j^2-1)2^{m-2}\quad\text{and}\quad
  \bigleftchoice{2^n|k2^{m-1}}{j=1}{2^n|k(j^{2^{m-1}}-1)/(j-1)}{j\neq 1},
\]
which we assume throughout.
For any $i\in \Z/2^n\Z$, we have
\begin{eqnarray}
  (c^ih)^{2^m} &=& c^{i(j^{2^m-1}+j^{2^m-2}+\ldots+j+1)}\, h^{2^m},
  \label{eq:orderch}\\ 
  (xh)^{2^m} &=& (c^{-k}h^2)^{2^{m-1}} =
  c^{-k(j^{2^m-2}+j^{2^m-4}+\ldots+j^2+1)}\,h^{2^m}.\label{eq:orderxh}
\end{eqnarray}

There are two conjugacy classes of non-trivial subgroups of $K$ that do not
contain $\centralC2$, represented by any $\langle c^\text{even}x\rangle$ and
any $\langle c^\text{odd}x\rangle$, respectively. 
\newline
\textbf{Case 1: $k$ even}\newline
Then, equation (\ref{eq:orderch}) implies that  $c^{-k/2}h$ has order $2^m$
and so without loss of generality, $h$ may be replaced by $c^{-k/2}h$,
which acts trivially on $x$ by conjugation.
Thus, assume without loss of generality that $k=0$. Also, equation (\ref{eq:orderxh})
implies that in this case, $xh$ has order $2^m$, so $h$ can always be replaced
by $xh$, which shows that for any $j$, the choices $j$ and $-j$ yield isomorphic
groups.

The elements of $\cH_m$ are isomorphic to $C_2\times C_{2^m}$. More precisely,
a general element of $\cH_m$ is of the form
$\langle c^\delta x, c^\gamma h\rangle$ with the two generators commuting and
with $c^\gamma h$ having order $2^m$.

We first claim that if $j=\pm1$, i.e. if $P=K\times A$, then $G$ has no
primitive relations. Indeed, suppose without loss of generality that $j=1$.
First note that
\begin{eqnarray*}
\langle c^\delta x, c^\gamma h\rangle\in\cH_m & \Longleftrightarrow &
(c^\delta x)(c^\gamma h) = (c^\gamma h)(c^\delta x)\\
& \Longleftrightarrow & c^{\delta-\gamma}xh = c^{\delta+\gamma}xh\\
& \Longleftrightarrow & \gamma\in\{0,2^{n-1}\}.
\end{eqnarray*}
Let $H_1=\langle c^{\delta_1} x, c^{\gamma_1} h\rangle$,
$H_2=\langle c^{\delta_2} x, c^{\gamma_2} h\rangle\in \cH_m$.
If $\delta_1=\delta_2$ (or indeed if they have the same parity), then $H_1$,
$H_2$ together generate a proper subgroup of $P$ and so lie in one connected
component of $\Gamma$ (see Theorem \ref{thm:Knontriv}, condition (1)).
If $\delta_1\neq\delta_2$ but $\gamma_1 = \gamma_2$, then $H_1\cap H_2$ has
index 2 in each of them and $H_1H_2/H_1\cap H_2$ is dihedral, so again they
lie in the same connected component (Theorem \ref{thm:Knontriv}, condition
(\ref{edge:quo})). Finally, if $\delta_1\neq\delta_2$ and $\gamma_1\neq
\gamma_2$, then
$H_3=\langle c^{\delta_2} x, c^{\gamma_1} h\rangle\in \cH_m$, also, and the
same argument can be applied to the pairs $H_1$, $H_3$ and $H_2$, $H_3$
to show that $H_1$ and $H_2$ lie in the same component of $\Gamma$.

Now, suppose that $j\neq \pm 1$. We will show that then, $\Prim(G)=C_2$.
Replacing $h$ by $xh$ if necessary, we may assume that $j\equiv 3\pmod 4$.
First, we claim that if $\langle c^\delta x, c^\gamma h\rangle\in\cH_m$,
then $\delta\equiv\gamma\pmod{2}$. Indeed, 
\begin{eqnarray*}
  \langle c^\delta x, c^\gamma h\rangle\in\cH_m & \Longleftrightarrow & c^\gamma h c^\delta x h^{-1} c^{-\gamma} = c^{2\gamma+\delta j}x = c^\delta x \\
  & \Longleftrightarrow & 2\gamma + \delta(j-1) \equiv 0 \pmod {2^n},
\end{eqnarray*}
which forces $\gamma$ and $\delta$ to have the same parity, since
$j-1\equiv 2\pmod 4$. Moreover, for any given $\gamma$, this equation has a
solution for $\delta$ and vice versa.

If $H_1=\langle c^{\delta_1} x, c^{\gamma_1} h\rangle$,
$H_2=\langle c^{\delta_2} x, c^{\gamma_2} h\rangle\in \cH_m$ satisfy
$\delta_1\equiv\delta_2\pmod 2$, and therefore also
$\gamma_1\equiv\gamma_2\pmod 2$, then $H_1$ and $H_2$ generate a proper
subgroup of $P$, so there is an edge between the corresponding vertices of
$\Gamma$. It follows that $\Gamma$ has at most
two connected components, represented by $H_{\text{odd}}$ with $\delta_1$ odd
and $H_{\text{even}}$ with $\delta_2$ even. It remains to show that two such
groups exist in $\cH_m$ and that they lie in distinct connected components of
$\Gamma$.
Now, there exists $H_{\text{odd}}\in \cH_m$ if and only if $c^{\text{odd}}h$
has order $2^m$, or equivalently (by (\ref{eq:orderch}))
$2^n|\frac{j^{2^m}-1}{j-1}$, or again equivalently (since $j\equiv 3\pmod{4}$)
$2^{n+1}|j^{2^m}-1$, or equivalently (by Lemma \ref{lem:ladicVal})
$2^{n+1}|(j^2-1)2^{m-1}$. This holds by assumption.

Clearly, $H_{\text{odd}}$ and $H_{\text{even}}$ together generate $P$, so it
is enough to show that $H_{\text{odd}}\cap H_{\text{even}}$ is of index 
greater than 2 in either subgroup. If the intersection was of index 2, then
it would contain an element of order $2^{m-1}$. But $H_{\text{odd}}$ has
only two elements of order $2^{m-1}$, namely
$e_1=(c^{\gamma_1}h)^2=c^{\gamma_1(j+1)}h^2$ and
$e_2=c^{\delta_1}x(c^{\gamma_1}h)^2=
c^{\delta_1-\gamma_1(j+1)}xh^2$, and similarly for $H_{\text{even}}$.
The former is in $H_{\text{even}}$ if and only if
\[
c^{\gamma_1(j+1)}h^2=c^{\gamma_2(j+1)}h^2\Leftrightarrow 2^n|(\gamma_1-\gamma_2)(j+1),
\]
which is impossible since $\gamma_1$, $\gamma_2$ have distinct parities and
$j$ is assumed to be not equal to -1 in $\Z/2^n\Z$.
Similarly, $e_2$ is in $H_{\text{even}}$ if and only if
\[
c^{\delta_1-\gamma_1(j+1)}xh^2=c^{\delta_2-\gamma_2(j+1)}xh^2\Leftrightarrow
2^n|\delta_1-\delta_2 + (\gamma_2-\gamma_1)(j+1),
\]
which is also impossible since $\delta_1-\delta_2$ is odd, while
$(\gamma_2-\gamma_1)(j+1)$ is even. This finishes the proof of case
(\ref{item:diheven}) of Theorem \ref{thm:main}.
\newline
\textbf{Case 2: $k$ is odd}
\newline
In this case, no complement of $C\times K$ in $P$ fixes a non-central
involution of~$K$.
\begin{notation}
Elements of $\cH_m$ are either isomorphic to $C_{2^m}$
and are complements of $C\times K$ in $P$, or isomorphic to
$C_2\times C_{2^{m-1}}$ and generated by the unique index 2 subgroup of a
complement of $C\times K$ in $P$ and by a non-central involution of $K$. We
call the two kinds of elements of $\cH_m$ \emph{full image subgroups} and
\emph{half image subgroups}, respectively, according to their image in
$\Aut C$.
\end{notation}

We begin by counting the number of elements of $\cH_m$ of each of the two
kinds. An arbitrary full image subgroup of $P$ is of the form
$H_g=\langle gh\rangle$ for $g\in K$, and conversely, such a group is in $\cH$
if and only if $gh$ has order $2^m$. By our assumptions (cf. (\ref{eq:orderch})),
this is satisfied for $g=c^i$ for any $i\in \Z/2^n\Z$. Also, by
(\ref{eq:orderxh}), $xh$ has order $2^m$ if and only if
\[
2^n|j^{2^m-2}+j^{2^m-4}+\ldots+j^2+1,
\]
which we assume to hold for $j=1$, while for $j\neq 1$ the condition is
equivalent to $m>n$ by Lemma \ref{lem:ladicVal}.
Since replacing $h$ by $c^ih$ for any $i$ changes neither $j$ nor the parity
of $k$, the same condition holds for $\langle x'h\rangle\in\cH$ for any
non-central involution $x'$ of $K$. It follows that there are $2^n$ full image
subgroups of $P$ if $m\leq n$, and $2^{n+1}$ otherwise.

Next, we determine which of these subgroups are conjugate in $P$.
The orbit of any full image subgroup $H$ under $P$-conjugation has size
$|P/N_PH|=|P/C_PH|$. The last equality follows from the fact that $H$ is a
complement in a semi-direct product. Since $H$ does not commute with any
non-central involution of $K$, we have
\[
C_PH=\{c^ah^b | b\in \Z/2^m\Z, a\in \Z/2^n\Z, a(j-1)=0\in\Z/2^n\Z\}.
\]
The centraliser has size $2^m(j-1,2^n)$, so the orbit has size
$\frac{2^{n+1}}{(j-1,2^n)}$. Note that if $m\leq n$, then our assumptions
imply that $j\equiv 3\pmod 4$. Indeed, otherwise
$v_2((j^{2^{m-1}}-1)/(j-1))=m-1<n$, contradicting the assumption that
$2^n|\frac{j^{2^{m-1}}-1}{j-1}$. Thus, if $m\leq n$, then there is one orbit
of full image subgroups of $P$.

If on the other hand $m>n$, then $H_x$ is also a full image subgroup
with centraliser
\[
C_PH_x = \{c^a(xh)^b | b\in \Z/2^m\Z,a\in \Z/2^n\Z, a(j+1)=0\in\Z/2^n\Z\}.
\]
Replacing $h$ by $xh$ if necessary, we may assume that $j\equiv 3\pmod 4$.
Then the orbit of $H=H_1$ has size $2^n$ and the other orbits have all size
$2^{n+1}/(j+1,2^n)$. We deduce that the total number of orbits is
\[
1+\frac{2^n}{2^{n+1}/(j+1,2^n)} = 1+\frac{(j+1,2^n)}{2}.
\]
From now on, assume without loss of generality that $j\equiv 3\pmod 4$.
To summarise the computations, we have
\begin{proposition}\label{prop:full}
If $m\leq n$, then $H_1$ represents the unique conjugacy class of
full image subgroups in $\cH_m$.
Otherwise, the distinct conjugacy classes of full image elements of
$\cH_m$ are represented by
\[
\{H_g\}, \qquad g \in \{1, x, cx, \ldots c^{(j+1,2^n)-1}x\}.
\]
\end{proposition}

We now turn to the computation of half image elements of $\cH_m$. These are
of the form
\[
  B = \langle c^b x^\delta h^2, c^a x \rangle.
\]

Multiplying $c^b x^\delta h^2$ by $c^a x$ if necessary
we may assume $\delta=0$, and denote the resulting group by $B_{a,b}$.
We have
\[
  c^b h^2 \cdot c^a x \cdot (c^b h^2)^{-1} \cdot (c^a x)^{-1}
   = c^{2b+k(j+1)+a(j^2-1)}.
\]
The two generators commuting is equivalent to
\begin{eqnarray}\label{eq:B}
  2b+k(j+1)+a(j^2-1) \equiv 0 \pmod{2^n},
\end{eqnarray}
which, given any $a\in\bZ/2^n\bZ$, has two solutions for $b$. This yields at
most $2^{n+1}$ half image subgroups. Moreover, since any solution of 
(\ref{eq:B}) automatically satisfies
\[
  v_2(2b) \ge v_2(j+1),
\]
we claim that $c^b h^2$ has order $2^{m-1}$ (the minimal possible), so that
any such $B_{a,b}$ really does represent an element of $\cH_m$. Indeed,
\[
  (c^b h^2)^{2^{m-1}} = c^{b(j^{2^m-2}+j^{2^m-4}+\ldots+j^2+1)}
     = \bigleftchoice{c^{b\frac{j^{2^m}-1}{j^2-1}}}{j^2\ne 1}
                  {c^{b2^{m-1}}}{j^2=1}.
\]
When $j^2\not\equiv 1\pmod {2^n}$,
\[
  b\, \frac{j^{2^m}-1}{j^2-1} = \frac{j^{2^{m-1}}-1}{j-1} \cdot \frac{b}{j+1} \cdot
      (j^{2^{m-1}}+1).
\]
The three terms on the right have 2-adic valuations $\ge v_2(2^n)$, $\ge -1$ and
$\ge 1$ respectively, so the product is a multiple of $2^n$.
When $j^2\equiv 1\pmod {2^n}$,
the same computation together with Lemma \ref{lem:ladicVal}
proves that $b2^{m-1}$ is a multiple of $2^n$, as required. Thus, the half
image subgroups of $P$ are precisely $B_{a,b}$ for any solution $a,b$ of
(\ref{eq:B}).

\begin{lem}
There are two conjugacy classes of half image maximal groups in $\cH$:
for any $a$ the representatives are $B_{a,b}, B_{a,b'}$, where $b,b'$
are the two different solutions of (\ref{eq:B}).
\end{lem}

\begin{proof}
Because conjugation by $h$ acts on the non-central involutions of $K$ as
$x\mapsto c^i x$, $i$ odd, and conjugation by $c$ maps $x\mapsto c^2x$,
all non-central involutions of $K$ are $P$-conjugate. Because the orbit
of $c^ax$ has length $2^n$, its centraliser has size $2^{m+1}$, so it must be
$\centralC2\times B_{a,b}$. Therefore the normaliser of $B_{a,b}$
is of the same size (a normalising element must fix
$B_{a,b}\cap K=\{1,c^ax\}$), so its orbit has size $2^n$ as well.
\end{proof}

Finally, we determine which vertices of $\Gamma$ corresponding to the elements
of $\cH_m$ are connected by an edge, and hence the structure of $\Prim(G)$.

\begin{lem}\label{lem:halffull}
If $j=-1$, then the unique full image subgroup of $P$ that intersects
a half image subgroup in an index 2 subgroup is $H$.
If $j\neq -1$, then there are no such full image subgroups.
\end{lem}

\begin{proof}
We always have the full image group $H=\langle h\rangle$. The unique
index 2 subgroup of $H$ is a subgroup of a half image group if and
only if $b=0$ is a solution of (\ref{eq:B}) (for some $a$),
which it is if and only if $j\equiv-1\pmod {2^n}$.

Other full image groups (in the case $m>n$) are generated by $c^dxh$
for some $d$.

The square of this element is $c^{d-k-dj}h^2$. The exponent $b$ is odd,
so it cannot solve (\ref{eq:B}), since $j+1\equiv 0\pmod 4$.
\end{proof}

\begin{proposition}
The group $\prim(G)$ is isomorphic to
\begin{itemize}
\item $\{1\}$ if $m\le n$ and $j=-1$;
\item $C_2$ if $m> n$ or $j\neq-1$ but not
both;
\item $C_2\times C_2$ if $m>n$ and
$j\neq-1$.
\end{itemize}
\end{proposition}
\begin{proof}
First, note that two half image groups always generate together at most
the preimage of the unique index 2 subgroup of $\Aut C$, so they always lie
in the same connected component of the graph $\Gamma$.
A full image group and a half image group together
always generate all of $P$ (the group they generate contains a non-central
involution of $K$ and a complement to $K$ in $P$, hence also $c$, since $k$
is odd). If $j=-1$, then $H$ intersects a half image subgroup in an index
2 subgroup and the quotient of $P$ by that intersection
is dihedral, so in this case, $H$ lies in the same connected component of
$\Gamma$ as the half image groups. 
Proposition \ref{prop:full} and Lemma \ref{lem:halffull} now imply that if
$m\le n$, then $\Prim(G)$ is trivial if $j=-1$ and has order 2 otherwise.
Suppose that $m>n$. Two full image groups of the form $H_{c^ix}$ generate
a proper subgroup of $P$, since the subgroup they generate contains no
non-central involutions of $K$. On the other hand, $H$ together with any
non-conjugate full image subgroup generates $P$ by the same argument as above,
and the intersection of two such groups has index greater than 2, using the
fact that $k$ is odd. The claim of the proposition follows.
\end{proof}
This finishes the proof of case (\ref{item:dihodd}) of Theorem \ref{thm:main}.

\begin{ackn}
The first author was
partially supported by the EPSRC and is partially supported by
a Research Fellowship from the Royal Commission for the Exhibition
of 1851, and the second author is supported by a Royal Society
University Research Fellowship. Parts of this research were done at 
St Johns College, Robinson College and DPMMS in Cambridge, 
CRM in Barcelona, and Postech University in Pohang. 
We would like to thank these institutions
for their hospitality and financial support.
\end{ackn}



\begin{thebibliography}{10}

\bibitem{bra1}
A. Bartel, T. Dokchitser, Brauer relations in finite groups, 2012,
to appear in J. Eur. Math. Soc., arXiv:1103.2047v4.

\bibitem{Bouc}
S. Bouc,
The Dade group of a $p$-group, Invent. Math. \textbf{164} no. 1 (2006), 189--231.


\bibitem{Gor-68}
W. Gorenstein, Finite Groups, Chelsea, 1968.


\end{thebibliography}
\end{document}